\documentclass{article}
\usepackage{array,amsmath,amsthm}
\numberwithin{equation}{section}
\usepackage{amssymb}
\usepackage{indentfirst}
\usepackage{geometry}
\begin{document}

\newtheorem{thm}{Theorem}[section]
\newtheorem{cor}[thm]{Corollary}
\newtheorem{prop}[thm]{Proposition}
\newtheorem{conj}[thm]{Conjecture}
\newtheorem{lem}[thm]{Lemma}
\newtheorem{Def}[thm]{Definition}
\newtheorem{rem}[thm]{Remark}
\newtheorem{prob}[thm]{Problem}
\newtheorem{ex}{Example}[section]

\newcommand{\be}{\begin{equation}}
\newcommand{\ee}{\end{equation}}
\newcommand{\ben}{\begin{enumerate}}
\newcommand{\een}{\end{enumerate}}
\newcommand{\beq}{\begin{eqnarray}}
\newcommand{\eeq}{\end{eqnarray}}
\newcommand{\beqn}{\begin{eqnarray*}}
\newcommand{\eeqn}{\end{eqnarray*}}
\newcommand{\bei}{\begin{itemize}}
\newcommand{\eei}{\end{itemize}}

\newcommand{\pa}{{\partial}}
\newcommand{\V}{{\rm V}}
\newcommand{\R}{{\bf R}}
\newcommand{\K}{{\rm K}}
\newcommand{\e}{{\epsilon}}
\newcommand{\tomega}{\tilde{\omega}}
\newcommand{\tOmega}{\tilde{Omega}}
\newcommand{\tR}{\tilde{R}}
\newcommand{\tB}{\tilde{B}}
\newcommand{\tGamma}{\tilde{\Gamma}}
\newcommand{\fa}{f_{\alpha}}
\newcommand{\fb}{f_{\beta}}
\newcommand{\faa}{f_{\alpha\alpha}}
\newcommand{\faaa}{f_{\alpha\alpha\alpha}}
\newcommand{\fab}{f_{\alpha\beta}}
\newcommand{\fabb}{f_{\alpha\beta\beta}}
\newcommand{\fbb}{f_{\beta\beta}}
\newcommand{\fbbb}{f_{\beta\beta\beta}}
\newcommand{\faab}{f_{\alpha\alpha\beta}}

\newcommand{\pxi}{ {\pa \over \pa x^i}}
\newcommand{\pxj}{ {\pa \over \pa x^j}}
\newcommand{\pxk}{ {\pa \over \pa x^k}}
\newcommand{\pyi}{ {\pa \over \pa y^i}}
\newcommand{\pyj}{ {\pa \over \pa y^j}}
\newcommand{\pyk}{ {\pa \over \pa y^k}}
\newcommand{\dxi}{{\delta \over \delta x^i}}
\newcommand{\dxj}{{\delta \over \delta x^j}}
\newcommand{\dxk}{{\delta \over \delta x^k}}

\newcommand{\px}{{\pa \over \pa x}}
\newcommand{\py}{{\pa \over \pa y}}
\newcommand{\pt}{{\pa \over \pa t}}
\newcommand{\ps}{{\pa \over \pa s}}
\newcommand{\pvi}{{\pa \over \pa v^i}}
\newcommand{\ty}{\tilde{y}}
\newcommand{\bGamma}{\bar{\Gamma}}

\title {Isoperimetric inequality on Finsler metric measure manifolds with non-negative weighted Ricci curvature\footnote{The first author is supported by the National Natural Science Foundation of China (12371051, 12141101, 11871126).}}
\author{Xinyue Cheng, \ Yalu Feng, \ Liulin Liu\\
School of Mathematical Sciences, Chongqing Normal University, \\ Chongqing, 401331, P.R. China\\
E-mail: chengxy@cqnu.edu.cn, fengyl2824@qq.com,\\ 2023010510010@stu.cqnu.edu.cn}

\date{}

\maketitle

\begin{abstract}
In this paper, we define the volume entropy and the second Cheeger constant and prove a sharp isoperimetric inequality involving the volume entropy on Finsler metric measure manifolds with non-negative weighted Ricci curvature ${\rm Ric}_{\infty}$. As an application, we prove a Cheeger-Buser type inequality for the first eigenvalue of Finsler Laplacian by using  the volume entropy and the second Cheeger constant.
\end{abstract}

\maketitle

\section{Introduction}\label{intr}

The isoperimetric inequality is one of the most important topics in geometry and has attracted extensive attention. Recent years have witnessed limited yet fundamental advances in extending the isoperimetric inequality to Finsler geometric settings. Zhao and Shen \cite{shenzhao} derived a Croke-type isoperimetric inequality for a relatively compact domain $\Omega$ with smooth boundary $\pa \Omega$ in  reversible Finsler manifolds equipped with the Busemann-Hausdorff volume form or the Holmes-Thompson volume form. They further proved that the equality holds if and only if $(\overline{\Omega}, F|_{\overline{\Omega}})$ is a hemisphere of a constant sectional curvature sphere. Later, Ohta \cite{ohtaloc} developed the needle decomposition (also called the localization) method on Finsler metric measure manifolds and derived a L\'{e}vy-Gromov type isoperimetric inequality on a complete Finsler metric measured manifold with finite reversibility  and finite volume. In subsequent work, Ohta \cite{ohtabl} proved the Bakry-Ledoux  isoperimetric inequality (also called the Gaussian isoperimetric inequality) on  compact Finsler metric measure manifolds satisfying ${\rm Ric}_{\infty}\geq K>0$, by using the Bochner inequality and the asymptotic behavior of (nonlinear or linearized) heat semigroups determined by a global solution to the heat equation.
More recently, by using optimal mass transport theory Balogh-Krist\'{a}ly \cite{BK}  proved a sharp isoperimetric inequality in $\mathrm{CD}(0, N) \, (N\in(1, \infty))$ metric measure spaces having Euclidean volume growth  in the sense that the asymptotic volume ratio
$$
\mathrm{AVR}_{(M, d, m)}:=\lim _{r \rightarrow \infty} \frac{m\left(B_r(x)\right)}{\omega_{N} r^{N}}>0,
$$
where $\omega_{N}(=\pi^{N/2} / \Gamma(1+N/2))$ denotes the volume of the Euclidean unit ball in $\mathbb{R}^{N}$ whenever $N\in \mathbb{N}$. In particular, the isoperimetric inequality holds on reversible Finsler metric measured manifolds with non-negative weighted Ricci curvature ${\rm Ric}_{N}$.
Furthermore, Manini \cite{manini} proved a sharp isoperimetric inequality on forward complete Finsler metric measure manifolds with non-negative weighted Ricci curvature ${\rm Ric}_{N}$ $\left( N\in(1, \infty)\right)$ having Euclidean volume growth. He also proved a rigidity result for his sharp isoperimetric inequality under the additional hypotheses of boundedness of the isoperimetric set and the finite reversibility of the space.

In this paper, we will first prove a sharp isoperimetric inequality on forward complete non-compact Finsler metric measure manifolds with non-negative weighted Ricci curvature ${\rm Ric}_{\infty}$. Our isoperimetric inequality involves a volume entropy, first introduced by Brooks \cite{brook} (also see Chapter 22 in \cite{peterli}), which plays a crucial role in geometric analysis. Further, as  applications, we will give new upper and lower bounds for the first eigenvalue of Finsler Laplacian in terms of the volume entropy and the second Cheeger constant.

In order to introduce our main results, we first give some necessary definitions and notations. We denote a Finsler manifold $(M, F)$ equipped with a positive smooth measure $\mathfrak{m}$ by $(M, F, \mathfrak{m})$ which we call a Finsler metric measure manifold. Such manifold is not a metric measure space in the usual sense since the Finsler metric $F$ may be non-reversible, that is, $F(x, y)\neq F(x, -y)$ may occur.
This non-reversibility causes the asymmetry of the associated distance function. In order to overcome this defect, Rademacher \cite{Ra} defined the reversibility $\Lambda_{F}$ of $F$ by
\be
\Lambda_{F}:=\sup _{(x, y) \in TM \backslash\{0\}} \frac{F(x,y)}{F(x, -y)}.
\ee
Obviously, $\Lambda_{F} \in [1, \infty]$ and $\Lambda_{F}=1$ if and only if $F$ is reversible. Later, Ohta \cite{Ohta} extended the concepts of uniform smoothness and uniform convexity in Banach space theory into Finsler geometry and gave their geometric interpretation. The uniform smoothness and uniform convexity of a Finsler metric $F$ mean that there exist two uniform constants $0<\kappa^{*}\leq 1 \leq \kappa <\infty$ such that for $x\in M$, $V\in T_xM\setminus \{0\}$ and $W\in T_xM$, we have
\begin{equation}
\kappa^{*}F^{2}(x, W)\leq g_{V}(W, W)\leq \kappa F^{2}(x, W), \label{usk}
\end{equation}
where $g_V$ is the weighted Riemann metric  induced by $V$ on the tangent bundle $TM$ of Finsler manifold $(M, F)$. If $F$ satisfies the uniform smoothness and uniform convexity, then $\Lambda_{F}$ is finite with
\be
1 \leq \Lambda_{F} \leq \min \left\{\sqrt{\kappa}, \sqrt{1 / \kappa^*}\right\}. \label{Lamba}
\ee
$F$ is Riemannian if and only if $\kappa=1$ if and only if $\kappa^{*}=1$ \cite{Ohta,OHTA,Ra}.

We define the forward and backward geodesic balls of radius $R$ with center at $x_{0}\in M$, respectively, as
$$
B_{R}^{+}(x_0):=\{x \in M \mid d(x_{0}, x)<R\},\qquad B_{R}^{-}(x_0):=\{x \in M \mid d(x, x_{0})<R\}.
$$
Given a subset $E\subset M$, the (forward) Minkowski exterior boundary measure (or Minkowski content) is defined as
\be
\mathfrak{m}^{+}(E):=\liminf\limits_{\varepsilon\rightarrow 0^{+}} \frac{\mathfrak{m}(B^{+}(E, \varepsilon))-\mathfrak{m}(E)}{\varepsilon}, \label{Mcontent}
\ee
where $B^{+}(E, \varepsilon):=\{z \in M \mid \inf_{x\in E}d(x, z)<\varepsilon\}$ is the forward $\varepsilon$-neighborhood of $E$ for $\varepsilon >0$.

Following the arguments in \cite{brook,peterli},  we define the volume entropy on Finsler metric measure manifolds as follows.
\begin{Def}\label{def-Log}
Let $(M, F, \mathfrak{m})$ be an $n$-dimensional Finsler metric measure manifold. The volume entropy is defined by
\be
{\rm VE}_{F}:= \lim\limits_{r\rightarrow\infty} \frac{\log \mathfrak{m}(B^{+}_{r}(x))}{r}\in [0, +\infty], \ \ x\in M.
\ee
\end{Def}
It is easy to see that the volume entropy ${\rm VE}_{F}$ is independent of the choice of $x\in M$ and it characterizes the exponential growth of the volume of $M$. From the definition of the volume entropy, we can prove the following isoperimetric inequality.

\begin{thm}\label{thm-isc}
Let $(M, F, \mathfrak{m})$ be an $n$-dimensional forward complete non-compact Finsler metric measure manifold satisfying ${\rm Ric_{\infty}}\geq 0$. Then for every bounded Borel subset $E\subset M$, it holds
\be
\mathfrak{m}^{+}(E)\geq {\rm VE}_{F} \cdot \mathfrak{m}(E).\label{eqn-isc}
\ee
Moreover, inequality  (\ref{eqn-isc}) is sharp.
\end{thm}

In 1970, Cheeger \cite{cheeger} introduced an isoperimetric constant $h(M)$ on a compact Riemannian manifold $M$ with or without boundary, now known as the Cheeger constant, and proved the Cheeger inequality, that is, $\lambda_{1}(M) \geq \frac{1}{4} h^{2}(M)$, where $\lambda_1$ is the smallest positive eigenvalue of the Laplacian. The Cheeger inequality revealed to be extremely useful in proving lower bounds on the first eigenvalue of the Laplacian in terms of the isoperimetric constant $h$. In Finsler geometry,  Shen \cite{Shen1} introduced a Cheeger constant on Finsler metric measure manifolds defined by
\[
\mathrm{I}_{\infty}(M):=\inf_{\Omega} \frac{\nu(\partial \Omega)}{\min\{\mathfrak{m}(\Omega), \mathfrak{m}(M)-\mathfrak{m}(\Omega)\}},
\]
where the infimum is taken over all regular domains $\Omega$ in $M$ and $\nu$ denotes volume measure on $\pa \Omega$ induced by $\mathfrak{m}$ with respect to a taken normal vector along $\pa \Omega$. In this paper, we naturally introduce  the second Cheeger constant on Finsler metric measure manifolds as follows.

\begin{Def}\label{def-ch}
Let $(M, F, \mathfrak{m})$ be an $n$-dimensional Finsler metric measure manifold. The second Cheeger constant on $(M, F, \mathfrak{m})$ is defined by
\be
{\rm SCh}_{F}:= \inf\limits_{E}\frac{\mathfrak{m}^{+}(E)}{\mathfrak{m}(E)},
\ee
where the infimum is taken over all bounded subset $E\subset M$.
\end{Def}

It is worth to mention that Buser \cite{buser} characterized the upper bound for the first eigenvalue of the Laplacian on Riemannian manifolds by using the isoperimetric constant $h$, the dimension of the manifolds and the lower bound on the Ricci curvature. More precisely, Buser  proved that for any closed $n$-dimensional Riemannian manifold satisfying ${\rm Ric}\geq K$ for some $K\leq0$, the following inequality, now known as Buser's inequality, holds
$$
\lambda_{1}(M) \leq 2 \sqrt{-(n-1) K} \, h(M) +10 h^{2}(M).
$$
Recently, Ponti-Mondino established some sharp Cheeger-Buser type inequalities for the first eigenvalue $\lambda_1$ of the Laplacian in ${\rm RCD}(K,\infty)$ spaces \cite{ponti1} and  Ponti-Mondino-Semola proved rigidity results for these equalities on ${\rm RCD}(K,\infty)$ spaces \cite{ponti2}.

Our following result devotes to determining the upper and lower bounds for the first eigenvalue of Finsler-Laplacian by using the second Cheeger constant ${\rm SCh}_{F}$ and the volume entropy ${\rm VE}_{F}$.

\begin{thm}\label{thm-lamd-L}
Let $(M, F, \mathfrak{m})$ be an $n$-dimensional complete non-compact Finsler metric measure manifold with uniform smoothness and uniform convexity. Then the first eigenvalue $\lambda_1$ satisfies
\beq
\frac{1}{4\kappa^{2}}{\rm SCh}_{F}^{2}\leq \lambda_{1}(M) \leq \frac{\kappa^2}{4}{{\rm VE}_{F}^2}. \label{inq-lamd-L}
\eeq
\end{thm}

Moreover, if $F$ satisfies ${\rm Ric}_{\infty}\geq 0$, the sharpness of (\ref{eqn-isc}) implies that  ${\rm SCh}_{F}={{\rm VE}_{F}}$. Hence we have the following theorem by Theorem \ref{thm-lamd-L}.

\begin{thm}
Let $(M, F, \mathfrak{m})$ be an $n$-dimensional complete non-compact Finsler metric measure manifold with uniform smoothness and uniform convexity. Assume that ${\rm Ric_{\infty}}\geq 0$. Then the first eigenvalue $\lambda_1$ satisfies
\beq
\frac{1}{4\kappa^{2}}{\rm SCh}_{F}^2\leq \lambda_{1}(M) \leq \frac{\kappa^2}{4}{\rm SCh}_{F}^2. \label{inq-lamd}
\eeq
\end{thm}

The inequality (\ref{inq-lamd}) can be regarded as the Cheeger-Buser type inequality for  the first eigenvalue of Finsler Laplacian in Finsler geometry.  When $F$ is a Riemannian metric ($\kappa=1$),   Cheeger-Buser type inequality (\ref{inq-lamd})  is sharp and implies that $\lambda_{1} =\frac{1}{4}{\rm SCh}_{F}^2$.  Naturally, we have the following important and interesting open problem in Finsler geometry.

\vskip 2mm
\noindent\textbf{Open problem:} Under what conditions about a Finsler metric measure manifold $(M,F,  \mathfrak{m})$, is the Cheeger-Buser type inequality  (\ref{inq-lamd})  sharp? Conversely, what geometric properties can we get for the Finsler metric measure  manifold $(M,F,  \mathfrak{m})$ from the sharpness of (\ref{inq-lamd})?

\vskip 2mm
This paper is organized as follows. In Section \ref{Introd}, we will give some necessary definitions, notations, and known results. Section \ref{Iso} is devoted to the isoperimetric inequality (\ref{eqn-isc}) and we will prove that it is sharp. Finally, we obtain the upper and lower bounds of the first eigenvalue and give the proof of Theorem \ref{thm-lamd-L} in Section \ref{Che}.

\section{Preliminaries}\label{Introd}
In this section, we briefly review some necessary definitions, notations and  fundamental results in Finsler geometry. For more details, please refer to \cite{BaoChern, ChernShen, OHTA, Shen1}.

\subsection{Finsler metric measure manifolds}

Let $M$ be an $n$-dimensional smooth manifold and  $TM$ be the tangent bundle on $M$. Denote the elements in $TM$ by $(x, y)$ with $y \in T_{x} M$. Let $T M_{0}:=TM \backslash\{0\}$ and $\pi: T M_{0} \rightarrow M$ be the natural projective map. The pull-back $\pi^{*} T M$ is an $n$-dimensional vector bundle on $T M_0$. A Finsler metric on manifold $M$ is a function $F: T M \longrightarrow[0, \infty)$  satisfying the following properties:
\ben
  \item [(1)] (Regularity) $F$ is $C^{\infty}$ on $TM_{0}$;
  \item [(2)] (Positive 1-homogeneity) $F(x,\lambda y)=\lambda F(x,y)$ for any $(x,y)\in TM$ and all $\lambda >0$;
  \item [(3)] (Strong convexity) $F$ is strongly convex, that is, the matrix $\left(g_{ij}(x,y)\right)=\left(\frac{1}{2}(F^{2})_{y^{i}y^{j}}\right)$ is positive definite for any nonzero $y\in T_{x}M$.
\een
Such a pair $(M,F)$ is called a Finsler manifold and $g:=g_{ij}(x,y)dx^{i}\otimes dx^{j}$ is called the fundamental tensor of $F$. For a non-vanishing vector field $V$ on $M$, one introduces the weighted Riemannian metric $g_V$ on $M$ given by
$$
g_V(y, w)=g_{ij}(x, V_x)y^i w^j
$$
for $y,\, w\in T_{x}M$. In particular, $g_{V}(V,V)=F^{2}(x, V)$ for any $V \in T_{x}M$.

We define the reverse metric $\overleftarrow{F} $ of $F$ by $\overleftarrow{F}(x, y):=F(x,-y)$ for all $(x, y) \in T M$. It is easy to see that $\overleftarrow{F}$ is also a Finsler metric on $M$. A Finsler metric $F$ on $M$ is said to be reversible if $\overleftarrow{F}(x, y)=F(x, y)$ for all $(x, y) \in T M$. Otherwise, we say $F$ is irreversible.

Given a Finsler structure $F$ on $M$,  there is a Finsler co-metric $F^{*}$ on $M$ which is non-negative function on the cotangent bundle $T^{*}M$ given by
\be
F^{*}(x, \xi):=\sup\limits_{y\in T_{x}M\setminus \{0\}} \frac{\xi (y)}{F(x,y)}, \ \ \forall \xi \in T^{*}_{x}M. \label{co-Finsler}
\ee
We call $F^{*}$ the dual Finsler metric of $F$.  For any vector $y\in T_{x}M\setminus \{0\}$, $x\in M$, the covector $\xi =g_{y}(y, \cdot)\in T^{*}_{x}M$ satisfies
\be
F(x,y)=F^{*}(x, \xi)=\frac{\xi (y)}{F(x,y)}. \label{shenF311}
\ee
Conversely, for any covector $\xi \in T_{x}^{*}M\setminus \{0\}$, there exists a unique vector $y\in T_{x}M\setminus \{0\}$ such that $\xi =g_{y}(y, \cdot)\in T^{*}_{x}M$ (Lemma 3.1.1, \cite{Shen1}). Naturally,  we define a map ${\cal L}: TM \rightarrow T^{*}M$ by
$$
{\cal L}(y):=\left\{
\begin{array}{ll}
g_{y}(y, \cdot), & y\neq 0, \\
0, & y=0.
\end{array} \right.
$$
It follows from (\ref{shenF311}) that
$$
F(x,y)=F^{*}(x, {\cal L}(y)).
$$
Thus ${\cal L}$ is a norm-preserving transformation. We call ${\cal L}$ the Legendre transformation on Finsler manifold $(M, F)$.

Let
$$
g^{*kl}(x,\xi):=\frac{1}{2}\left[F^{*2}\right]_{\xi _{k}\xi_{l}}(x,\xi).
$$
For any $\xi ={\cal L}(y)$, we have
$$
g^{*kl}(x,\xi)=g^{kl}(x,y),
$$
where $\left(g^{kl}(x,y)\right)= \left(g_{kl}(x,y)\right)^{-1}$. If $F$ is uniformly smooth and uniformly convex with (\ref{usk}), then  $\left(g^{*ij}\right)$ is uniformly elliptic in the sense that there exist two constants $\tilde{\kappa}=(\kappa^*)^{-1}$, $\tilde{\kappa}^*=\kappa^{-1}$ such that for $x \in M, \ \xi \in T^*_x M \backslash\{0\}$ and $\eta \in T_x^* M$, we have \cite{OHTA}
\be
\tilde{\kappa}^* F^{* 2}(x, \eta) \leq g^{*i j}(x, \xi) \eta_i \eta_j \leq \tilde{\kappa} F^{* 2}(x, \eta). \label{unisc}
\ee

Given a smooth function $u$ on $M$, the differential $d u_x$ at any point $x \in M$, $d u_x=\frac{\partial u}{\partial x^i}(x) d x^i$ is a linear function on $T_x M$. We define the gradient vector $\nabla u(x)$ of $u$ at $x \in M$ by $\nabla u(x):=\mathcal{L}^{-1}(d u(x)) \in T_x M$. In a local coordinate system, we can express $\nabla u$ as
\be \label{nabna}
\nabla u(x)= \begin{cases}g^{* i j}(x, d u) \frac{\partial u}{\partial x^i} \frac{\partial}{\partial x^j}, & x \in M_u, \\ 0, & x \in M \backslash M_u,\end{cases}
\ee
where $M_{u}:=\{x \in M \mid d u(x) \neq 0\}$ \cite{Shen1}. In general, $\nabla u$ is only continuous on $M$, but smooth on $M_{u}$. Further,  for $u \in W^{1,2}(M)$, noting that $\nabla u$ is weakly differentiable, the Finsler Laplacian  $\Delta u$ is defined by
\be
\int_M \phi \Delta u d m:=-\int_{M} d \phi(\nabla u) dm  \label{Lap1}
\ee
for $\phi \in \mathcal{C}_{0}^{\infty}(M)$, where ${\cal C}^{\infty}_{0} (M)$ denotes the set of all smooth compactly supported functions on $M$  \cite{Shen1}.

Given a weakly differentiable function $u$ and a vector field $V$ which does not vanish on $M_u$, the weighted gradient and the weighted Laplacian of $u$ on the weighted Riemannian manifold $\left(M, g_{V}, \mathfrak{m}\right)$ are defined  by
$$
\nabla^V u:= \begin{cases}g^{ij}(x, V) \frac{\partial u}{\partial x^i} \frac{\partial}{\partial x^j} & \text { for } x \in M_u, \\ 0 & \text { for } x \notin M_u , \end{cases} \ \ \ \Delta^{V} u:= {\rm div}_{\mathfrak{m}}\left(\nabla^V u\right).
$$
We note that $\nabla^{\nabla u}u=\nabla u$ and $\Delta^{\nabla u} u= \Delta u$ \cite{OHTA}.

Let $(M, F, \mathfrak{m})$ be a Finsler metric measure manifold, $\Omega \subset M$ a domain with compact closure and nonempty boundary $\partial \Omega$. The first eigenvalue $\lambda_1(\Omega)$ of Finsler Laplacian on $\Omega$ is defined by {\cite{Shen1}}
$$
\lambda_1(\Omega)=\inf _{u \in H_{0}^{1,2}(\Omega)\setminus \{0\}} \frac{\int_{\Omega}\left(F^{*}(d u)\right)^2 d \mathfrak{m}}{\int_{\Omega} u^2 d \mathfrak{m}},
$$
where $H_{0}^{1,2}(\Omega)$ is the completion of $C_{0}^{\infty}(\Omega)$ in $W^{1,2}(\Omega)$ with respect to the norm
$$
\|u\|_{\Omega}=\|u\|_{L^{2}(\Omega)}+\frac{1}{2}\|F^{*}(d u)\|_{L^{2}(\Omega)}+\frac{1}{2}\|\overleftarrow{F}^{*}(d u)\|_{L^{2}(\Omega)}.
$$

If $\Omega_1 \subset \Omega_2$ are bounded domains, then $\lambda_1\left(\Omega_1\right) \geq \lambda_2\left(\Omega_2\right) \geq 0$. Thus, if $\Omega_1 \subset \Omega_2 \subset$ $\cdots \subset M$ are bounded domains so that $\cup \ \Omega_i=M$, then the following limit

$$
\lambda_1(M)=\lim _{i \rightarrow \infty} \lambda_1\left(\Omega_i\right) \geq 0
$$
exists, and it is independent of the choice of $\left\{\Omega_i\right\}$.

\vskip 2mm

For $x_1, x_2 \in M$, the distance from $x_1$ to $x_2$ is defined by
$$
d\left(x_1, x_2\right):=\inf _\gamma \int_0^1 F(\dot{\gamma}(t)) d t,
$$
where the infimum is taken over all $C^1$ curves $\gamma:[0,1] \rightarrow M$ such that $\gamma(0)=$ $x_1$ and $\gamma(1)=x_2$. Note that $d \left(x_1, x_2\right) \neq d \left(x_2, x_1\right)$ unless $F$ is reversible. Let $E \subset M$ be a bounded Borel set. Define the diameter of $E$ by
\[
{\rm diam} (E):=\sup\limits_{x_{1}, x_{2} \in E}\{d(x_{1}, x_{2})\}.
\]

A $C^{\infty}$-curve $\gamma:[0,1] \rightarrow M$ is called a geodesic  if $F(\gamma, \dot{\gamma})$ is constant and it is locally minimizing.
The exponential map $\exp _x: T_x M \rightarrow M$ is defined by $\exp _x(v)=\gamma(1)$ for $v \in T_x M$ if there is a geodesic $\gamma:[0,1] \rightarrow M$ with $\gamma(0)=x$ and $\dot{\gamma}(0)=v$. A Finsler manifold $(M, F)$ is said to be forward complete (resp. backward complete) if each geodesic defined on $[0, \ell)$ (resp. $(-\ell, 0])$ can be extended to a geodesic defined on $[0, \infty)$ (resp. $(-\infty, 0])$. We say $(M, F)$ is complete if it is both forward complete and backward complete. By Hopf-Rinow theorem on forward complete Finsler manifolds, any two points in $M$ can be connected by a minimal forward geodesic and the forward closed balls $\overline{B_R^{+}(p)}$ are compact (see \cite{BaoChern, Shen1}).  If the reversibility constant $\Lambda_F$ is finite, then the forward completeness is equivalent to the backward completeness.
In this case, we can simply speak of completeness without ambiguity \cite{OHTA}.

\subsection{${\rm Ric}_{\infty}\geq 0$ or ${\rm CD}(0, \infty)$ condition}

Let $(M, F, \mathfrak{m})$ be  an $n$-dimensional Finsler manifold $(M, F)$ equipped with a smooth measure $\mathfrak{m}$. Write the volume form $d \mathfrak{m}$ of  $\mathfrak{m}$ as $d \mathfrak{m} = \sigma(x) dx^{1} dx^{2} \cdots d x^{n}$. Define
\be\label{Dis}
\tau (x, y):=\ln \frac{\sqrt{{\rm det}\left(g_{i j}(x, y)\right)}}{\sigma(x)}.
\ee
We call $\tau$ the distortion of $F$. It is natural to study the rate of change of the distortion along geodesics. For a vector $y \in T_{x} M \backslash\{0\}$, let $ \gamma = \gamma(t)$ be the geodesic with $\gamma (0)=x$ and $\dot{\gamma}(0)=y.$  Set
\be
{\bf S}(x, y):= \frac{d}{d t}\left[\tau(\gamma (t), \dot{\gamma}(t))\right]|_{t=0}.
\ee
$\mathbf{S}$ is called the S-curvature of $F$ \cite{ChernShen}. Let $Y$ be a $C^{\infty}$ geodesic field on an open subset $U \subset M$ (i.e. all integral curves of $Y$ on $U$ are geodesics) and $\hat{g}=g_{Y}$ .  Let
\be
d \mathfrak{m}:=e^{- \psi} {\rm Vol}_{\hat{g}}, \ \ \ {\rm Vol}_{\hat{g}}= \sqrt{{det}\left(g_{i j}\left(x, Y_{x}\right)\right)}dx^{1} \cdots dx^{n}. \label{voldecom}
\ee
It is easy to see that $\psi$ is given by
$$
\psi (x)= \ln \frac{\sqrt{\operatorname{det}\left(g_{i j}\left(x, Y_{x}\right)\right)}}{\sigma(x)}=\tau\left(x, Y_{x}\right),
$$
which is just the distortion of $F$ along $Y_{x}$ at $x\in M$ \cite{ChernShen, Shen1}. Let $y := Y_{x}\in T_{x}M$ (that is, $Y$ is a geodesic extension of $y\in T_{x}M$). Then, by the definitions of the S-curvature, we have
\beqn
&&  {\bf S}(x, y)= Y[\tau(x, Y)]|_{x} = d \psi (y),  \\
&&  \dot{\bf S}(x, y)= Y[{\bf S}(x, Y)]|_{x} =y[Y(\psi)],
\eeqn
where $\dot{\bf S}(x, y):={\bf S}_{|m}(x, y)y^{m}$ and ``$|$" denotes the horizontal covariant derivative with respect to the Chern connection \cite{shen, Shen1}. Further, Ohta introduces the weighted Ricci curvatures in Finsler geometry in \cite{Ohta0} which are defined as follows \cite{Ohta0,OHTA}
\beq
{\rm Ric}_{N}(y)&=& {\rm Ric}(y)+ \dot{\bf S}(x, y) -\frac{{\bf S}(x, y)^{2}}{N-n},   \label{weRicci3}\\
{\rm Ric}_{\infty}(y)&=& {\rm Ric}(y)+ \dot{\bf S}(x, y). \label{weRicciinf}
\eeq
We say that Ric$_{N}\geq K$ for some $K\in \mathbb{R}$ if ${\rm Ric}_{N}(v)\geq KF^{2}(v)$ for all $v\in TM$, where $N\in \mathbb{R}\setminus \{n\}$ or $N= \infty$.

\vskip 3mm

Given a forward complete Finsler metric measure manifold $(M, F, \mathfrak{m})$. We denote by $\mathcal{P}(M)$ the space of all Borel probability measures on $M$.
For $p \in[1, \infty)$, we consider the space $\mathcal{P}^{p}(M) \subset \mathcal{P}(M)$ consisting of the measures  with finite $p$-th moment, that is, $\mu \in \mathcal{P}^{p}(M)$ if we have
$$
\int_{M}\{d^{p}(o, x)+d^{p}(x, o)\} \mu(d x)<\infty
$$
for some (hence for any) $ o \in M $.  For $\mu_{0}, \mu_{1}\in \mathcal{P}^{p}(M)$,  a probability measure $\pi \in \mathcal{P}(M \times M)$ on the product space is called  a coupling of $(\mu_{0}, \mu_{1})$ if  $\pi$  satisfies
$$
\pi(A \times M)=\mu_{0}(A), \quad \pi(M \times A)=\mu _{1}(A)
$$
for any measurable set $A \subset M$. We denote by $\Pi(\mu_{0}, \mu_{1})$ the set of all couplings of $(\mu_{0}, \mu_{1})$.
Now we define the $L^{p}$-Wasserstein distance $W_{p}$ from $\mu_0$ to $\mu_1$ by
\be
W_{p}\left(\mu_{0}, \mu_{1}\right):=\inf _{\pi \in \Pi(\mu_{0}, \mu_{1})}  \left(\int_{M \times M} d^{p}(x, z) \pi(d x d z)\right)^{\frac{1}{p}}. \label{Wadis}
\ee
We call $\left(\mathcal{P}^{p}(M), W_{p}\right)$ the $L^p$-Wasserstein space over $M$.  The  problem  to find and characterize a coupling $\pi$ of $(\mu_{0}, \mu_{1})$ attaining the infimum  in (\ref{Wadis}) is called the Monge-Kantorovich problem \cite{manini,OHTA}.

We call a  curve $\mu:[0,1] \rightarrow \mathcal{P}^{p}(M)$  in the Wasserstein space $\left(\mathcal{P}^{p}(M), W_p\right)$ the geodesic in $\left(\mathcal{P}^{p}(M), W_p\right)$  if $\mu$ satisfies
$$
W_{p}\left(\mu_{t}, \mu_s\right)=(s-t) W_{p}\left(\mu_{0}, \mu_{1}\right), \quad \forall \ 0 \leq t \leq s \leq 1 ,
$$
where $\mu_{t}:= \mu (t)$ for $t \in [0,1]$. It can be shown that if $\mu_0$ and $\mu_1$ are absolutely continuous, there exists a unique geodesic connecting $\mu_0$ to $\mu_1$  \cite{manini,OHTA}.
The relative entropy of an absolutely continuous measure $\mu=\rho \mathfrak{m}\in \mathcal{P}^{2}(M)$ is defined by
$$
{\rm Ent}_{\mathfrak{m}}(\mu):=\int_{M} \rho \log \rho \ d\mathfrak{m},
$$
where $\rho$ is the density function (the Radon-Nikodym derivative) of $\mu$ with respect to $\mathfrak{m}$.

The $\mathrm{CD}(K, N)$ condition for metric measure spaces was introduced in the seminal works of Sturm \cite{st1, st2} and Lott-Villani \cite{lot}, and later investigated in the realm of measured Finsler manifolds by Ohta (please refer to Chapter 18 of \cite{OHTA}). In particular, Ohta gave the following definition.

\begin{Def}[Curvature-Dimension Condition $\mathrm{CD}(K, \infty)$, \cite{OHTA}, Definition 18.5]
Let $(M, F, \mathfrak{m})$ be a forward complete Finsler metric measure manifold. We say that $(M, F, \mathfrak{m})$ satisfies the curvature-dimension condition  $\mathrm{CD}(K, \infty)$ for $K \in \mathbb{R}$ if, for any pair $\mu_{i}=\rho_{i} \mathfrak{m} \in \mathcal{P}^{2}(M) \ (i=0,1)$, we have
$$
\operatorname{Ent}_{\mathfrak{m}}\left(\mu_\lambda\right) \leq(1-\lambda) \operatorname{Ent}_{\mathfrak{m}}\left(\mu_0\right)+\lambda \operatorname{Ent}_{\mathfrak{m}}\left(\mu_1\right)-\frac{K}{2}(1-\lambda) \lambda W_{2}^{2}\left(\mu_{0}, \mu_{1}\right) .
$$
In particular, we say that $(M, F, \mathfrak{m})$ satisfies the curvature-dimension condition $\mathrm{CD}(0, \infty)$ if the relative entropy is convex along the geodesic of the Wasserstein space, that is, for all couples of absolutely continuous curves $\mu_{0}, \mu_{1}\in \mathcal{P}^{2}(M)$, it holds that
\be
{\rm Ent}_{\mathfrak{m}} (\mu_{t})\leq (1-t){\rm Ent}_{\mathfrak{m}} (\mu_{0})+t{\rm Ent}_{\mathfrak{m}} (\mu_{1}), \label{CD-con}
\ee
where $(\mu_{t})_{t\in[0,1]}$ is the unique geodesic connecting $\mu_{0}$ to $\mu_{1}$.
\end{Def}

\begin{thm}[ \cite{OHTA}, Theorem 18.6]
Let $(M, F, \mathfrak{m})$ be a forward or backward complete measured Finsler manifold of dimension $n \geq 2$, and take $K \in \mathbb{R}$. Then ${\rm Ric}_{\infty} \geq K$ holds if and only if $(M, F, \mathfrak{m})$ satisfies the curvature-dimension condition $\mathrm{CD}(K, \infty)$.
\end{thm}

A natural and important result derived from the curvature-dimension condition  $\mathrm{CD}(K, \infty)$ (or ${\rm Ric}_{\infty} \geq K$)  is the Brunn-Minkowski inequality. Given two measurable subsets $A$ and $B$ of a $\mathrm{CD}(K, \infty)$ Finsler metric measure manifold $(M, F, \mathfrak{m})$ and $t \in[0,1]$, we define
$$
\begin{aligned}
Z_{t}(A, B) :&=\left\{\gamma_{t} \mid \gamma \text { is a minimal geodesic such that } \gamma_{0} \in A \text { and } \gamma_{1} \in B\right\} \\
& =\{z \mid \exists \ x_{0} \in A, z_{0} \in B \ \text{such that} \ \mathrm{d}(x_{0}, z)=t \mathrm{d}(x_{0}, z_{0}) \text{ and } \mathrm{d}(z, z_{0})=(1-t) \mathrm{d}(x_{0}, z_{0})\}.
\end{aligned}
$$
An almost immediate consequence of the above definition is the Brunn-Minkowski inequality on Finsler metric measure manifolds satisfying the ${\rm CD}(K, \infty)$ condition.
\begin{lem} [Brunn-Minkowski Inequality, \cite{OHTA}, Theorem 18.8] Let $(M, F, \mathfrak{m})$  be forward or backward complete and satisfy $\operatorname{Ric}_{\infty} \geq K$ for some $K \in \mathbb{R}$. Take two Borel sets $A, B \subset M$ with positive masses. Then we have
$$
\begin{aligned}
& \log \mathfrak{m}\left(Z_{t}(A, B)\right) \\
& \geq (1-t) \log \mathfrak{m}(A)+t \log \mathfrak{m}(B)+\frac{K}{2}t (1-t)  W_{2}^{2}\left(\frac{\mathbf{1}_A}{\mathfrak{m}(A)} \mathfrak{m}, \frac{\mathbf{1}_B}{\mathfrak{m}(B)} \mathfrak{m}\right)
\end{aligned}
$$
for all $t \in (0,1)$, where $\mathbf{1}_A$ denotes the characteristic function of $A$. In particular, if $(M, F, \mathfrak{m})$ satisfies  $\operatorname{Ric}_{\infty} \geq 0$, we have
\be
\log \mathfrak{m}\left(Z_t(A, B)\right)\geq (1-t) \log \mathfrak{m}(A)+t \log \mathfrak{m}(B), \quad t \in(0,1). \label{B-M-ine}
\ee
\end{lem}

\section{Isoperimetric inequality}\label{Iso}

This section is devoted to the proof of  Theorem \ref{thm-isc}.

\begin{proof}[Proof of Theorem \ref{thm-isc}]
Let $E \subset M$ be a bounded Borel set and $d:= {\rm diam}(E)$. Obviously, $d <\infty$. Let $x_{0}\in E$ and fix arbitrarily a $R>0$. Then, for any $t\in [0,1]$, we claim that
\be
Z_{t}(E, B^{+}_{R}(x_0))\subset B^{+}(E, t(d+R)), \label{eqn-zt}
\ee
where $B^{+}(E, \varepsilon)$ is the forward $\varepsilon$-neighborhood of $E$ for $\varepsilon >0$. Indeed, if  $z\in Z_{t}(E, B^{+}_{R}(x_0))$, then by the definition, there exist $x\in E$ and $z_{0}\in B^{+}_{R}(x_0)$ such that $d(x,z)=td(x, z_{0})$. By the triangular inequality, we have that $d(E, z)\leq d(x, z)= td(x, z_{0})\leq t(d(x, x_{0})+d(x_{0}, z_{0}))<t(d+R)$. Thus, $z\in B^{+}(E, t(d+R))$, which means (\ref{eqn-zt}).

From (\ref{eqn-zt}) and Brunn-Minkowski inequality (\ref{B-M-ine}), we have that
\be
\log \mathfrak{m}(B^{+}(E, t(d+R)))\geq\log \mathfrak{m}(Z_{t}(E, B^{+}_{R}(x_0)))\geq (1-t)\log \mathfrak{m}(E)+t\log \mathfrak{m}(B^{+}_{R}(x_0)).\label{eqn-mt}
\ee

Now we are in the position to compute the Minkowski content $\mathfrak{m}^{+}(E)$. We just consider the case that $\mathfrak{m}^{+}(E)<\infty$, otherwise (\ref{eqn-isc}) is trivial. This implies $\lim\limits_{\varepsilon\rightarrow 0}\mathfrak{m}(B^{+}(E, \varepsilon))=\mathfrak{m}(E)$. Then, by the derivative formula $(s^{\frac{1}{n}})^{\prime}|_{s= s_{0}}= \frac{1}{n}s_{0}^{\frac{1}{n}-1}$, we find that
$$
\mathfrak{m}(E)^{\frac{1}{n}-1}=\lim\limits_{\varepsilon\rightarrow 0^{+}} \ n \cdot \frac{\mathfrak{m}(B^{+}(E, \varepsilon))^{\frac{1}{n}}-\mathfrak{m}(E)^{\frac{1}{n}}}{\mathfrak{m}(B^{+}(E, \varepsilon))-\mathfrak{m}(E)}.
$$
Thus, by using the definition of the Minkowski content we have that
$$
\frac{\mathfrak{m}^{+}(E)}{\mathfrak{m}(E)^{1-\frac{1}{n}}}
= \liminf\limits_{\varepsilon\rightarrow 0^{+}} \ n \frac{\mathfrak{m}(B^{+}(E, \varepsilon))^{\frac{1}{n}}-\mathfrak{m}(E)^{\frac{1}{n}}}{\varepsilon}.
$$
Replacing $\varepsilon$ with $t(d+R)$ and by (\ref{eqn-mt}), we obtain that
$$
\frac{\mathfrak{m}^{+}(E)}{\mathfrak{m}(E)^{1-\frac{1}{n}}}\geq\liminf\limits_{t\rightarrow 0^{+}} \ n \frac{\mathfrak{m}(E)^{\frac{1-t}{n}} \cdot \mathfrak{m}(B^{+}_{R}(x_0))^{\frac{t}{n}}-\mathfrak{m}(E))^{\frac{1}{n}}}{t(d+R)}.
$$
By L'H\^{o}pital's rule, we have
$$
\liminf\limits_{t\rightarrow 0^{+}} \ n \frac{\mathfrak{m}(E)^{\frac{1-t}{n}} \cdot \mathfrak{m}(B^{+}_{R}(x_0))^{\frac{t}{n}}-\mathfrak{m}(E))^{\frac{1}{n}}}{t(d+R)}=\mathfrak{m}(E)^{\frac{1}{n}} \frac{\log \mathfrak{m}(B^{+}_{R}(x_0))-\log\mathfrak{m}(E)}{d+R}.
$$
Thus,
$$
\frac{\mathfrak{m}^{+}(E)}{\mathfrak{m}(E)}\geq \frac{\log \mathfrak{m}(B^{+}_{R}(x_0))-\log\mathfrak{m}(E)}{d+R}.
$$
By the arbitrariness of $R>0$ and letting $R\rightarrow\infty$ on the right hand side, we can obtain $\mathfrak{m}^{+}(E)\geq {\rm VE}_{F} \cdot \mathfrak{m}(E)$.

Next, we prove the sharpness of (\ref{eqn-isc}) by contradiction. Assume that there is a constant $C>{\rm VE}_{F}$ such that
\beq
\mathfrak{m}^{+}(E)\geq C \cdot \mathfrak{m}(E)\label{ine-sh1}
\eeq
for all bounded Borel subset $E\subset M$. To obtain the contradiction, we fix arbitrarily a $R>0$ and choose any $0<r<R$. For any $\delta>0$ and $x_{0}\in M$, if $z\in Z_{t}(B^{+}_{r}(x_{0}), B^{+}_{R+\delta}(x_{0}))$, then by the definition, there exist $x_{1}\in B^{+}_{r}(x_{0})$ and $z_{1}\in B^{+}_{R+\delta}(x_{0})$ such that $d(x_{1}, z)=t d(x_{1}, z_{1})$, and then
$$
d(x_{0}, z)\leq d(x_{0}, x_{1})+d(x_{1}, z)=d(x_{0}, x_{1})+ t d(x_{1}, z_{1})<r+ t(d+ R+\delta),
$$
where $d:={\rm diam}(B^{+}_{r}(x_{0}))$. Hence, for $t=\frac{R-r}{R+d+\delta}$, we have $Z_{t}(B^{+}_{r}(x_{0}), B^{+}_{R+\delta}(x_{0}))\subset B^{+}_{R}(x_{0})$. Then it holds from Brunn-Minkowski inequality (\ref{B-M-ine}) that
\beq
\log \mathfrak{m}(B^{+}_{R}(x_{0}))\geq \frac{r+d+\delta}{R+d+\delta}\log \mathfrak{m}(B^{+}_{r}(x_{0}))+ \frac{R-r}{R+d+\delta} \log \mathfrak{m}(B^{+}_{R+\delta}(x_{0})).\label{ine-sh2}
\eeq
Multiplying both sides of (\ref{ine-sh2}) by $\frac{R+d+\delta}{R-r}$, we have
\beq
\log \mathfrak{m}(B^{+}_{R+\delta}(x_{0}))-\log \mathfrak{m}(B^{+}_{R}(x_{0})) \leq \frac{r+d+\delta}{R-r}\left(\log \mathfrak{m}(B^{+}_{R}(x_{0}))-\log \mathfrak{m}(B^{+}_{r}(x_{0}))\right).\label{ine-sh3}
\eeq
Since $B^{+}(B^{+}_{R}(x_{0}), \delta)\subset B^{+}_{R+\delta}(x_{0})$, then, by using the definition of the Minkowski content and the derivative formula $(\log s)^{\prime}|_{s= s_{0}}= \frac{1}{s_{0}}$, we have the following
\beqn
C \leq \frac{\mathfrak{m}^{+}(B^{+}_{R}(x_{0}))}{\mathfrak{m}(B^{+}_{R}(x_{0}))}&=& \liminf\limits_{\delta\rightarrow 0^{+}} \frac{\log \mathfrak{m}(B^{+}(B^{+}_{R}(x_{0}), \delta))-\log \mathfrak{m}(B^{+}_{R}(x_{0}))}{\mathfrak{m}(B^{+}(B^{+}_{R}(x_{0}), \delta))-\mathfrak{m}(B^{+}_{R}(x_{0}))} \cdot  \frac{ \mathfrak{m}(B^{+}(B^{+}_{R}(x_{0}),\delta))-\mathfrak{m}(B^{+}_{R}(x_{0}))}{\delta}\\
&\leq &\liminf\limits_{\delta\rightarrow 0^{+}} \frac{\log \mathfrak{m}(B^{+}_{R+\delta}(x_{0}))-\log \mathfrak{m}(B^{+}_{R}(x_{0}))}{\delta}.
\eeqn
Then there exists a $\delta>0$ such that
\begin{equation}\label{ine-sh4}
\delta  C \leq \log \mathfrak{m}(B^{+}_{R+\delta}(x_{0}))-\log \mathfrak{m}(B^{+}_{R}(x_{0})).
\end{equation}
Combining (\ref{ine-sh3}) with (\ref{ine-sh4}) yields
\begin{equation*}
C \leq \frac{r+d+\delta}{\delta} \ \frac{\log \mathfrak{m}(B^{+}_{R}(x_{0}))-\log \mathfrak{m}(B^{+}_{r}(x_{0}))}{R-r}.
\end{equation*}
Due to the arbitrariness of $R>0$ and $0<r<R$, by letting $R\rightarrow\infty$ firstly and then letting $r\rightarrow 0^{+}$ (which implies $d\rightarrow 0^{+}$) in the above inequality, we can conclude that
$$
C \leq {\rm VE}_{F}<C,
$$
which is a contradiction. Then we have proven the sharpness of (\ref{eqn-isc}). This completes the proof.
\end{proof}

\section{Second Cheeger constant and Cheeger-Buser type inequality}\label{Che}
In this section, we  will prove Theorem \ref{thm-lamd-L}. In order to determine the lower bound of the first eigenvalue of Finsler Laplacian, we need the following co-area formula with respect to $\mathfrak{m}^{+}(E)$ for $E\subset M$ (please refer to the proof of Lemma B.2 and Lemma B.3 in \cite{manini}).

\begin{lem}{\rm (\cite{manini})}\label{coarea}
Let $(M, F,\mathfrak{m})$ be an $n$-dimensional Finsler metric measure manifold. If $f\in C_{0}^{1}(M)$, it holds that
\be
\int_{0}^{\infty}\mathfrak{m}^{+}(\{f\geq t\}) dt \leq \int_{M} F^{*}(-df)d\mathfrak{m}. \label{co-area}
\ee
\end{lem}

\vskip 2mm

By Lemma \ref{coarea}, we can first prove the following result.

\begin{lem}\label{lower}
Let $(M,F, \mathfrak{m})$ be an $n$-dimensional complete non-compact Finsler metric measure manifold with finite reversibility $\Lambda_{F}$. Then the first eigenvalue satisfies
\be
\lambda_1(M)\geq \frac{1}{4\Lambda_{F}^{2}}{\rm SCh}_{F}^{2}.\label{lower-1}
\ee
\end{lem}
\begin{proof}
Let $\Omega \subset M$ be a domain with compact closure and nonempty boundary, and ${\rm SCh}_{F}(\Omega):=\inf\limits_{E} \frac{\mathfrak{m}^{+}(E)}{\mathfrak{m}(E)}$ for all bounded subset $E\subset \Omega$. For any $f\in C_{0}^{\infty}(\Omega)$, we can decompose $f$ into the sum of the positive and negative parts, $f=f_{+}+f_{-}$, where $f_{+}:=\max\{f,0\}$, $f_{-}:=\min\{f,0\}$. Then by Lemma \ref{co-area}, we have
\beqn
{\rm SCh}_{F}(\Omega)\int_{\Omega}f^{2}d\mathfrak{m}&=&{\rm SCh}_{F}(\Omega)\int_{\Omega}(f_{+}^{2}+f_{-}^{2})d\mathfrak{m}\\
&=& {\rm SCh}_{F}(\Omega)\left(\int_{\Omega}\int^{\infty}_{0}\mathbf{1}_{\{x\in\Omega|f_{+}^{2}\geq t\}}(x)dtd\mathfrak{m}+\int_{\Omega}\int^{\infty}_{0}\mathbf{1}_{\{x\in\Omega|f_{-}^{2}\geq t\}}(x)dtd\mathfrak{m}\right)\\
&=&{\rm SCh}_{F}(\Omega) \left(\int^{\infty}_{0}\mathfrak{m}(\{x\in\Omega|f_{+}^{2}\geq t\})dt+ \int^{\infty}_{0}\mathfrak{m}(\{x\in\Omega|f_{-}^{2}\geq t\})dt\right)\\
&\leq & \int^{\infty}_{0} \mathfrak{m}^{+}(\{x\in\Omega|f_{+}^{2}\geq t\})dt+\int^{\infty}_{0} \mathfrak{m}^{+}(\{x\in\Omega|f_{-}^{2}\geq t\})dt\\
&\leq & \int_{\Omega}F^{*}\left(-d(f_{+}^{2})\right)d\mathfrak{m}+\int_{\Omega}F^{*}\left(-d(f_{-}^{2})\right)d\mathfrak{m}\\
&\leq & 2\Lambda_{F} \left(\int_{\Omega}f_{+}F^{*}(df_{+})d\mathfrak{m}+\int_{\Omega}(-f_{-})F^{*}(df_{-})d\mathfrak{m}\right)\\
&\leq & 2\Lambda_{F} \int_{\Omega} |f| F^{*}(df)d\mathfrak{m} \leq 2\Lambda_{F} \left(\int_{\Omega} f^{2}d\mathfrak{m}\right)^{\frac{1}{2}}\left(\int_{\Omega} F^{*}(df)^{2}d\mathfrak{m}\right)^{\frac{1}{2}},
\eeqn
where $\mathbf{1}_{A}$ denotes the characteristic function of set $A$. Thus, we obtain that
$$
\lambda_{1}(\Omega)\geq \frac{1}{4\Lambda_{F}^{2}}{\rm SCh}_{F}(\Omega)^{2}\geq \frac{1}{4\Lambda_{F}^{2}}{\rm SCh}_{F}^{2}.
$$
Let $\Omega:=B^{+}_{R}(x_0)$ for any $x_{0}\in M$ and $R>0$. Then by letting $R\rightarrow\infty$, we get (\ref{lower-1}).
\end{proof}

In the following, we focus on determining the upper bound of the first eigenvalue of Finsler Laplacian.  Inspired by the argument of Lemma 24.3 in \cite{peterli}, we can prove the following result.

\begin{lem}\label{upper}
Let $(M,F, \mathfrak{m})$ be an $n$-dimensional  complete non-compact Finsler metric measure manifold with uniform smoothness and uniform convexity. Then the first eigenvalue satisfies
\be
\lambda_1(M)\leq \frac{\kappa^2}{4}{\rm VE}_{F}^{2}. \label{upper-M}
\ee
\end{lem}
\begin{proof}
We will mainly prove that for any $0<\delta<1$, $R>2$ and $x_0\in M$,
\beq
\lambda_1(B^{+}_{2R}(x_0))\leq \frac{\kappa^2}{4\delta^2(R-1)^2}\left(\log\frac{\mathfrak{m}(B^{+}_{R}(x_0))}{\mathfrak{m}(B^{+}_{1}(x_0))}+\frac{\delta}{2\kappa}+\log\frac{8\kappa(2\kappa\tilde{\kappa}+\delta)}{\kappa-\delta^2}\right)^2.\label{upper-del}
\eeq
Obviously,  by letting first $R \rightarrow +\infty$ and then letting $\delta\rightarrow 1$ on the both sides of inequality (\ref{upper-del}), we can obtain (\ref{upper-M}).

For the simplicity  of notation, we will use $\lambda_1$ to denote $\lambda_1(B^{+}_{2R}(x_0))$. We may assume that $\lambda_{1}\geq \frac{1}{4 R^{2}}$ (otherwise the conclusion (\ref{upper-del}) automatically holds). Let $r(x)=d(x_{0}, x)$ denote the distant function starting from a fixed point $x_{0}\in M$. The variation principle implies that
\beq
\lambda_1 \int_{B^{+}_{2R}(x_0)} \phi^2 {\rm e}^{-2\delta \kappa^{-1}\sqrt{\lambda_1}r(x)}d\mathfrak{m} &\leq& \int_{B^{+}_{2R}(x_0)} F^{*2}\left(d(\phi {\rm e}^{-\delta \kappa^{-1}\sqrt{\lambda_1}r(x)})\right)d\mathfrak{m} \nonumber\\
&\leq& \kappa \int_{B^{+}_{2R}(x_0)} d(\phi {\rm e}^{-\delta \kappa^{-1}\sqrt{\lambda_1}r(x)})\left(\nabla^{\nabla r}(\phi {\rm e}^{-\delta \kappa^{-1}\sqrt{\lambda_1}r(x)})\right) d\mathfrak{m}\nonumber\\
&\leq& \kappa \int_{B^{+}_{2R}(x_0)} {\rm e}^{-2\delta \kappa^{-1}\sqrt{\lambda_1}r(x)} d\phi(\nabla^{\nabla r}\phi)d\mathfrak{m}\nonumber\\
& &-2\delta \sqrt{\lambda_1} \int_{B^{+}_{2R}(x_0)} \phi {\rm e}^{-2\delta \kappa^{-1}\sqrt{\lambda_1} r(x)} d\phi(\nabla r)d\mathfrak{m}\nonumber\\
& & + \delta^2 \lambda_1 \kappa^{-1}\int_{B^{+}_{2R}(x_0)} \phi^2{\rm e}^{-2\delta \kappa^{-1}\sqrt{\lambda_1}r(x)} d\mathfrak{m}\label{upper-lambda}
\eeq
for any non-negative function $\phi\in C_{0}^{\infty}(B^{+}_{2R}(x_0))$, where we have used (\ref{unisc}) in the second inequality. For any $R>2$, we choose a cut off function $\phi$ as
$$
\phi=\begin{cases} 1, & B^{+}_{R-\frac{1}{4\sqrt{\lambda_{1}}}}(x_{0}), \\ 4\sqrt{\lambda_{1}}(R-r(x)), & B^{+}_{R}(x_{0}) \setminus B^{+}_{R-\frac{1}{4\sqrt{\lambda_{1}}}}(x_{0}), \\ 0, &  B^{+}_{2R}(x_0) \setminus B^{+}_{R}(x_{0}).\end{cases}
$$
Then $F^{*}(-d \phi) \leq 4\sqrt{\lambda_{1}}$ a.e. on $B^{+}_{R}(x_{0})$. Thus, from (\ref{upper-lambda}) we can obtain the following inequality.
\beqn
\lambda_{1} (1-\delta^2\kappa^{-1}){\rm e}^{-2\delta \kappa^{-1}\sqrt{\lambda_1}}\mathfrak{m}(B^{+}_{1}(x_{0}))
&\leq&\lambda_1 (1-\delta^2\kappa^{-1})\int_{B^{+}_{2R}(x_0)} \phi^2 {\rm e}^{-2\delta \kappa^{-1}\sqrt{\lambda_1}r(x)}d\mathfrak{m}\\
&\leq& \kappa \int_{B^{+}_{2R}(x_0)} {\rm e}^{-2\delta \kappa^{-1}\sqrt{\lambda_1}r(x)} d\phi(\nabla^{\nabla r}\phi)d\mathfrak{m}\\
  & &-2\delta \sqrt{\lambda_1} \int_{B^{+}_{2R}(x_0)} {\rm e}^{-2\delta \kappa^{-1}\sqrt{\lambda_1}r(x)} d\phi(\nabla r)d\mathfrak{m}.
\eeqn
By (\ref{unisc}), we have $d\phi(\nabla^{\nabla r}\phi)= -d\phi(\nabla^{\nabla r}(-\phi))\leq \tilde{\kappa} F^{*2}(-d\phi)$. Furthermore, we can get the following
\beqn
\lambda_{1} (1-\delta^2\kappa^{-1}){\rm e}^{-2\delta \kappa^{-1}\sqrt{\lambda_1}}\mathfrak{m}(B^{+}_{1}(x_{0}))
&\leq& \kappa \tilde{\kappa}\int_{B^{+}_{2R}(x_0)} {\rm e}^{-2\delta \kappa^{-1}\sqrt{\lambda_1}r(x)} F^{*2}(-d\phi) d\mathfrak{m}\\
  & &+2\delta \sqrt{\lambda_1} \int_{B^{+}_{2R}(x_0)} {\rm e}^{-2\delta \kappa^{-1}\sqrt{\lambda_1}r(x)} F^{*}(-d\phi)d\mathfrak{m}\\
&\leq& 8(2\kappa \tilde{\kappa} +\delta)\lambda_1 \int_{B^{+}_{R}(x_{0}) \setminus B^{+}_{R-\frac{1}{4\sqrt{\lambda_{1}}}}(x_{0})} {\rm e}^{-2\delta \kappa^{-1} \sqrt{\lambda_1}\left(R-\frac{1}{4\sqrt{\lambda_1}}\right)} d\mathfrak{m}\\
&\leq& 8(2\kappa \tilde{\kappa} +\delta)\lambda_1  {\rm e}^{-2\delta \kappa^{-1} \sqrt{\lambda_1}\left(R-\frac{1}{4\sqrt{\lambda_1}}\right)} \mathfrak{m}(B^{+}_{R}(x_{0})),
\eeqn
from which we can obtain
$$
\lambda_1(B^{+}_{2R}(x_0))\leq \frac{\kappa^2}{4\delta^2(R-1)^2}\left(\log\frac{\mathfrak{m}(B^{+}_{R}(x_0))}{\mathfrak{m}(B^{+}_{1}(x_0))}+\frac{\delta}{2\kappa}+ \log\frac{8\kappa(2\kappa\tilde{\kappa}+\delta)}{\kappa-\delta^2}\right)^{2},
$$
which is just (\ref{upper-del}). This completes the proof of Lemma \ref{upper}.
\end{proof}

\vskip 2mm

\begin{proof}[Proof of Theorem \ref{thm-lamd-L}]
Theorem \ref{thm-lamd-L} directly follows from Lemma \ref{lower} and Lemma \ref{upper}.
\end{proof}

\end{document}